\documentclass[12pt]{amsart} 

\usepackage{amscd, latexsym, amsmath, amssymb, longtable, booktabs}

\newtheorem{theorem}{Theorem}%[section] (If  you want theorem numbered
\newtheorem{lemma}[theorem]{Lemma}%   with   section   number.    Same
\newtheorem{corollary}{Corollary}%    goes     for    lemmas,    etc.)
\newtheorem{proposition}[theorem]{Proposition}

\theoremstyle{definition}                 
 \newtheorem{defn}{Definition}
\newtheorem*{notation}{Notation}
\theoremstyle{remark} \newtheorem*{remark}{Remark}

\usepackage{amsfonts}%     to     get     the     /mathbb     alphabet
\newcommand{\field}[1]{\mathbb{#1}}          \newcommand{\Q}{\field{Q}}
\newcommand{\R}{\field{R}}                   \newcommand{\Z}{\field{Z}}
\newcommand{\C}{\field{C}}

\renewcommand{\a}{\alpha}
\renewcommand{\b}{\beta}                      
                    \newcommand{\e}{\varepsilon}
\newcommand{\g}{\gamma}                         \newcommand{\G}{\Gamma}
\renewcommand{\l}{\lambda}

\renewcommand{\th}{\theta}

 \newcommand{\ra}{\rightarrow}

\begin{document}

\title[Finiteness of Monodromy]
{Sums of Fractions and Finiteness of Monodromy}

\author{E. Ghate and T. N. Venkataramana}
\address{School of Mathematics, Tata Institute of Fundamental Research, 
Homi Bhabha Road, Mumbai 400005, India}
\email{eghate@math.tifr.res.in, venky@math.tifr.res.in}

\date{}

\begin{abstract}
 We solve an elementary number theory problem on sums of fractional parts, 
 using methods from group theory.
 We apply our result to deduce the finiteness of certain 
 monodromy representations. 
\end{abstract}

\maketitle

\section{Introduction}

In this  paper, we are  concerned with an elementary  number theoretic
question on a  sum of certain fractional parts.  The simplest instance
of this  is when there are  only three fractional  parts involved, and
the  classification  of such  $3$-tuples  is  equivalent to  Schwarz's
classification of  algebraic Euler-Gauss hypergeometric  functions. We
give a different proof  of the Schwarz classification using elementary
considerations, as well as use the Schwarz classification to show that
the  number theoretic  condition  does  not hold  when  the number  of
fractional parts  is more than  six, and show that it holds only  sporadically when
the  number  of  fractional  parts   is  four  or  five  (see  Theorem
\ref{numbertheoretic}). \\

It  turns out that  the answer  to the  aforementioned question  is closely
connected  to  the  finiteness  of  certain monodromy  groups.   As  a
consequence  of our  main result  on  fractional parts,  we classify when 
the  image of  certain specializations  of  the so called Gassner
representation is finite.  By linking these  specializations with  the monodromy
representations associated to certain  families of cyclic coverings of
the projective line of the  type considered by Deligne and Mostow (see
\cite{Del-Mos}), we recover results of Cohen and Wolfart \cite{Coh-Wol} 
on finiteness of
monodromy groups. Another  corollary of  the  main  result on  fractional  parts is  the
algebraicity of  certain Lauricella $F_D$-type  functions, also proved
in \cite{Coh-Wol}  (see also \cite{Ssk} and  \cite{Bod}) by completely
different methods.\\

We now go into some detail. First, we introduce some notation. 

\begin{notation} Let $d\geq  2$ and $n\geq 2$ be  integers.  Fix $n+1$
integers  $1\leq k_i  \leq  d-1$ such  that  the g.c.d.   of $d,  k_1,
\cdots,k_{n+1}$  is  $1$.   Given  $s$  in  the  multiplicative  group
$(\Z/d\Z)^*$ of  units of $\Z/d\Z$,  consider the numbers  $\mu _i(s)$
(denoted $\mu _i=\{\frac{k_i}{d} \}$ when $s=1$) defined by
\[\mu_i(s)= \left\{\frac{k_is}{d}\right\},\] where  $0 \leq \{x\} < 1$
denotes  the fractional  part  of a  real  number $x$.   We may  write
$k_is=q_id+l_i$, where $1\leq  l_i \leq d-1$ and $q_i$  is an integer;
thus     the    remainder    $l_i$     has    the     property    that
$\{\frac{k_is}{d}\}=\frac{l_i}{d}$. If we denote by $[x]$ the integral
part   of  $x$,  then   $x=[x]+\{x\}$.   The   number  $\mu   _i  (s)=
\{\frac{k_is}{d}\}$ depends  only on the  fraction $\frac{k_i}{d} =\mu
_i$.
\end{notation}

\begin{defn} 
We say that the rational  numbers $\mu _1, \cdots, \mu _{n+1}$ satisfy
the condition (\ref{SS}) if,
\begin{equation} \label{SS}
\begin{gathered}
\forall  { s} \in (\Z/d\Z)^*, \\
\text{either} \quad \sum _{i=1}^{n+1}  \left\{\frac  {k_i s}{d}  \right\}<1
\quad \text{or} \quad \sum _{i=1}^{n+1} \left\{-\frac{k_is}{d} \right\} < 1.
\end{gathered}
\end{equation}

In terms  of the  remainders $l_i$ above, this  means that
either $\sum l_i <d$ or else $\sum (d-l_i) <d$.
\end{defn}

The  following theorem says  that condition (\ref{SS})  on $n,d,
k_i$ is  very stringent and holds  in a very limited  number of cases.
Let us say that the tuple $(\frac{k_1}{d}, \cdots, \frac{k_{n+1}}{d})$
is   {\it   equivalent}   to   the  tuple   $(\frac{l_1}{d},   \cdots,
\frac{l_{n+1}}{d})$, if  there exists $t  \in (\Z/d\Z) ^*$  such that,
for all $i$, we  have $\{\frac{l_i}{d}\}= \{\frac{k_it}{d}\}$, up to a
permutation of the indices. The validity of condition (1) depends only 
on  the  equivalence  class  of  the  tuple  $(\frac{k_1}{d},  \cdots,
\frac{k_{n+1}}{d})$.

\begin{theorem} \label{numbertheoretic}  Suppose $n,d,k_i$ are  as in
the preceding so  that condition (\ref{SS}) holds.  Then 
\[n\leq 4.\] 

Moreover, up to equivalence, the numbers 
$\mu_1, \cdots, \mu_{n+1}$ satisfy the conditions given below. \\

\begin{enumerate}
\item[(i)] If $n=4$, then $\mu _i=\frac{1}{6}$, for all $i\leq n+1=5$. \\

\item[(ii)] If $n=3$, then there are only two cases: 
$\mu _i =\frac{1}{6}$, for all $i\leq n+1=4$, or 
$\mu _1=\mu _2=\mu _3= \frac{1}{6}$ and $\mu _4= \frac{2}{6}$. \\

\item[(iii)]  If $n=2$,  then either we may write $\mu _i=\frac{k_i}{d}$ with 
$d=2m$, for $m \geq 1$, and $k_1 = k_2 = p$, $k_3=m-p$, with $1 \leq p \leq m-1$ coprime to $m$, 
or else, the $\mu _i=\frac{k_i}{d}$ lie in a finite list with $d\leq 60$.
\end{enumerate}
\end{theorem}

\begin{remark} Note  that condition (\ref{SS}) is  a purely number
theoretic condition; the proof of the theorem, however, will depend on
an analysis of certain finite subgroups of unitary groups generated by
reflections.
\end{remark}

The  proof of  Theorem \ref{numbertheoretic}  proceeds as  follows. In
Section \ref{Sectionn=2},  we prove the theorem  for $n=2$. %Though the  result 
%is purely number  theoretic, 
We link condition (\ref{SS}) in the case $n = 2$ 
to the finiteness of a certain subgroup of the  unitary group of an
explicit skew-Hermitian form  (implicitly, this is the monodromy group
of the Gauss hypergeometric function, but we do not use this). In Section~\ref{ngeq3}, 
we use  a bootstrapping argument to show that condition (\ref{SS})
holds in very few cases for $n=3$ and $4$. Using this it is finally 
shown that condition (\ref{SS}) cannot hold for $n \geq 5$. \\

In  Section~\ref{skewhermitian},  we  show  that a  slightly  modified
version  of condition  (\ref{SS}), namely  condition (\ref{*})  in the
text,  is   equivalent  to  the   total  anisotropy  of   an  explicit
skew-Hermitian  form  in  $n$  variables  over  the  cyclotomic  field
$E=\Q(e^{2\pi i  \frac{1}{d}})$.  Conditions (\ref{SS})  and (\ref{*})
coincide for $n = 2$, and we show that they are equivalent for general
$n$.  We deduce that the image of the Gassner representation at $d$-th
roots of  unity is finite if  and only if  condition (\ref{SS}) holds,
providing  us  with  the algebraicity results  on  monodromy  groups mentioned
above (see Theorem~\ref{numbertheoreticmonodromy} in  Section~\ref{monodromy} below). As is (more or  less) known, the
finiteness of the image of the Gassner representation is equivalent to
the algebraicity  of the associated Lauricella  $F_D$-functions and we
list     some     of    these     results     as    corollaries     in
Section~\ref{lauricella}. \\

{\bf Acknowledgments}  We thank Paula Tretkoff and  Jurgen Wolfart for
making  their paper  (\cite{Coh-Wol}) available  to us,  and  also for
pointing out very helpful  references (especially \cite{Bod}). We also
thank F. Beukers and  E. Looijenga for interesting conversations about
the material of  the paper.  That the existence  of a totally definite
Hermitian  form is equivalent  to finiteness  of monodromy  is already
implicit in \cite{CHL}  and we thank Looijenga for  mentioning this to
one of us. \\

\noindent T. N. Venkataramana  gratefully acknowledges the  support of the
J. C. Bose fellowship for the year 2013-2018. 

\section{The case $n=2$.} \label{Sectionn=2}

\subsection{Definition}

Let    $d, k_1, k_2, k_3$     be    positive    integers     such    that
$d\Z+k_1\Z+k_2\Z+k_3\Z=\Z$. We say that these integers  
satisfy condition (\ref{S}) if,

\begin{equation} 
  \label{S}
  \begin{gathered}
  \text{for all} \> s \in (\Z/d\Z)^*,  \\
  \text{either}  \quad   \Sigma_s \stackrel{def}{=} \sum _{j=1}^3\left\{\frac{k_js}{d} \right\} <1 
  \quad \text{or} \quad \Sigma_{-s} =  \sum _{j=1}^3\left\{-\frac{k_js}{d}\right\}<1.
  \end{gathered}
\end{equation}

\begin{remark}  

[0] Condition (\ref{S}) is just condition (\ref{SS}) for $n = 2$.  \\

[1] Condition  (\ref{S})  depends  only  on  the
fractional parts $\nu _j = \{\frac{k_j}{d}\}$ of $\frac{k_j}{d}$, for $j=1,2,3$, 
and not directly on the  numbers $(d,k_1,k_2,k_3)$; for example,  condition (\ref{S})
holds   for  $(d, k_1,k_2,k_3)$   if   and  only   if   it  holds   for
$(d, k_1+d,k_2,k_3)$, etc. \\

[2] We  may also permute  the integers $k_1,k_2,k_3$  without changing
 condition (\ref{S}). \\

[3] If  $(d, k_1,k_2,k_3)$ is replaced by  $(d, k_1t,k_2t,k_3t)$ for some
integer $t$ coprime to $d$, then condition (\ref{S}) is unaltered.\\

[4]  Since $\{-x  \}=1-\{x\}$, for a real number $x$, condition (\ref{S})  
%for  $\Sigma _s$ 
is equivalent to saying  that 
%if $\Sigma$ is the sum $\sum _{j=1}^3 \{ \frac{k_js}{d}\}$, then  
either $0<\Sigma_s <1$ or $2 < \Sigma_s <3$, for each sum $\Sigma_s$.  
That is, the  integral part of each sum  $\Sigma_s$ is either $0$
or $2$ (but not $1$).

\end{remark}

\subsection{Main result for triples} We say that a  triple of  rational numbers 
$(\nu _1, \nu _2,  \nu _3)$ as above is  equivalent to another such triple 
$(\nu  _1',\nu _2',\nu _3  ')$ (for the same denominator $d)$, if there exists  
$t\in (\Z/d\Z)^*$ such  that, after a permutation  of the indices,    
we   have   $\nu    _j=\{\frac{k_j}{d}\}$   and    $   \nu
_j'=\{\frac{k_jt}{d}\}$, for $j=1,2,3$. By the remarks in the preceding
subsection, if condition (\ref{S}) holds for one triple, then it holds for 
all equivalent triples. \\

For $d$ and $k_1, k_2, k_3$ as above, write 
\begin{eqnarray*}
  \lambda & = & 1-\left\{\frac{k_1}{d}\right\}-\left\{\frac{k_2}{d}\right\}, \\
     \mu  & = & 1-\left\{\frac{k_1}{d}\right\}-\left\{\frac{k_3}{d}\right\}, \\                           
     \nu  & = & 1-\left\{\frac{k_2}{d}\right\}-\left\{\frac{k_3}{d}\right\}. 
\end{eqnarray*}   
If   $(\frac{k_1}{d}, \frac{k_2}{d}, \frac{k_3}{d})$  satisfy condition (\ref{S}), we may assume that 
$0 < \lambda, \mu, \nu <1$.

\begin{theorem} \label{n=2} (The case $n = 2$) If  $(d, k_1,k_2,k_3)$ satisfy condition
(\ref{S}), then up to the foregoing equivalence, we have either

\[ \quad \frac{k_1}{d}=\frac{k_2}{d}=\frac{p}{2m}, 
\quad {\rm and} \quad  \frac{k_3}{d}=\frac{m-p}{2m},\]
for some $m \geq 1$ and some $1 \leq p < m$ coprime to $m$, so that 
$$\lambda = \frac{m-p}{m}, \> \mu = \nu = \frac{1}{2}$$ 
(we refer to this as the ``dihedral case''), or
else
\[ \quad (\lambda, \mu ,\nu) \in \text{the finite list in Table 1 below}. \] 

\end{theorem}

\begin{remark} Again, though  the statement of the theorem is  purely (elementary)  number
theoretic, the proof uses the finiteness of a certain group
$\Gamma  $ in $GL_2(\C)$.  It would  be interesting  to find  a purely
number theoretic proof of the above theorem.
\end{remark}

\subsection{Relation of Condition (\ref{S}) with a skew-Hermitian form}

\begin{notation}

Let $E/F$ be a totally imaginary quadratic extension of a totally real
number field. Then $E=F[t]/(t^2+\a)$ for some totally positive element
$a$ in the  real subfield  $F$. $E/F$ is called a CM extension. Denote by  $z\mapsto {\overline  z} \quad
(\forall z\in E)$ the action  of the non-trivial element of the Galois
group  of $E/F$, induced by complex conjugation (under any embedding of $E$ into 
$\C$).   Let $h:E^n\times  E^n\ra E$,  denoted $(x,y)\mapsto
h(x,y)$  be an  $F$-bilinear  form which  is $E$-linear  in  the first
variable $x$  and such that  for all $x,y\in  E^n$, $h(y,x)=-\overline
{h(x,y)}$. Then $h$ is called a skew-Hermitian form on $E^n$. \\

If we  replace $F$ by $\R$ and  $E$ by $\C$, a skew-Hermitian form can
still be defined and it is of the form $h(x,y)=iH(x,y)$ where $H$ is a
Hermitian form on $\C^n$. \\

We say that a skew-Hermitian form $h$ on $E^n$ is {\bf anisotropic}, 
if $h(x,x)=0$, for $x \in E^n$, implies that $x=0$. Over $\C/\R$, a skew-Hermitian 
form $h$ is anisotropic if   and   only    if   $h=\pm   i   H$,   where    $H$   
is Hermitian and positive definite.  Furthermore, a diagonal  skew-Hermitian form over $\C/\R$ is
anisotropic if and only if the diagonal entries are $\l _1, \cdots, \l
_n$, with $\l _j\in i \R  \setminus \{0\}$ being on the imaginary axis,
and such that  the successive ratios $\l _{j+1}/\l  _j$ are positive
real numbers.  \\

We say that a skew-Hermitian form $h$ defined over $E/F$ is {\bf totally anisotropic} if it is 
anisotropic over $\C/\R$, for all embeddings of $E$ into $\C$, or more precisely
for all archimedean places of $F$ into $\R$. Note that for   
a skew-Hermitian form $h$ defined over $E$, anisotropy over $\C$ implies anisotropy over $E$, 
but the converse does not hold. 
\end{notation}

Now let $d$ and $k_1, k_2, k_3$ be as above. Write $x_j=e^{2\pi i \frac{k_j}{d}}$, for $j=1,2,3$. 
Let $E = \Q(e^{\frac{2\pi i}{d}})$ be the $d$-th cyclotomic  field, and let  
$F  =  \Q(\cos(\frac{2\pi}{d}))$  be the maximal totally real subfield of $E$.    \\

The matrix  
\[h= \begin{pmatrix} \frac{1-x_1x_2}{(1-x_1)(1-x_2)} & -\frac{x_2}{1-x_2} \\ -\frac{1}{1-x_2} & \frac{1-x_2x_3}{(1-x_2)(1-x_3)}  \end{pmatrix}  \] 
is easily seen to define a skew-Hermitian form over $E/F$, i.e., $^t h= - {\overline h}$.
% where the bar refers to complex conjugation). 
The determinant $\det (h)$ of $h$ is also easily computed to be 

\[\frac{1-x_1x_2x_3}{(1-x_1)(1-x_2)(1-x_3)}=- \frac{1}{4} \cdot
\frac{\sin (\frac{\pi (k_1+k_2+k_3)}{d})}
{\sin(\frac{\pi k_1}{d}) \sin (\frac{\pi k_2}{d}) \sin (\frac{\pi k_3}{d})} \in F. \]

%Note that $\det(h) \in F$.

\begin{lemma} \label{Sandhermitian} We have:  
\begin{enumerate}
  \item [i)] The skew-Hermitian form $h$ is
totally anisotropic if and  only if $\det(h)$ is a totally
negative element of $F$. 

 \item [ii)] The  numbers $\frac{k_j}{d}$, for $j = 1,2,3$,  satisfy the  condition (\ref{S})  if and
only if the skew-Hermitian form $h$ is totally anisotropic.
\end{enumerate}
\end{lemma} 

\begin{proof} Fix an embedding of $E$ into $\C$. 
The Gram-Schmidt process says that under this embedding $h$ is equivalent to
the skew-Hermitian form  $h'=\begin{pmatrix} i\lambda _1  & 0 \\  0 &
i\lambda   _2\end{pmatrix}$  for  some   real  numbers $\lambda   _1,  \lambda
_2$. Moreover, the principal minors of $h$ and $h'$ are the same: $\det
(h)=\det(h')$ and $i\lambda _1=\frac{1-x_1x_2}{(1-x_1)(1-x_2)}$. \\

The form $h$ is anisotropic if and only if the equivalent form $h'$
is  anisotropic, and  the  latter  holds   if  and  only  if   the  fraction
$\frac{\lambda _2}{\lambda _1}$ is positive. This fraction may also be
written as
\[\frac{(i\lambda    _1)(i\lambda    _2)}{(i\lambda   _1)^2}=\frac{\det
(h)}{-\lambda  _1^2} .\]  Thus  $h$ is anisotropic if and  only if
$\det (h)$ is negative. This argument is independent of the 
embedding of the field $E$ into $\C$ and  hence $h$ is totally anisotropic 
if and only if  its determinant  is totally negative.  This proves  the first
part of the lemma. \\

If $t\in \R \setminus \Z$, it  is easily seen that the sign of $\sin
(\pi  t)$  is  $(-1)^{[t]}$,  where  $[t]$ is  the  integral  part  of
$t$. Therefore, by the paragraph  preceding the statement of the lemma, the sign of the
determinant of $h$ is seen to be
\[- (-1)^{[\frac{k_1+k_2+k_3}{d}]
-[\frac{k_1}{d}]-[\frac{k_2}{d}]-[\frac{k_3}{d}]}.\]   Now,   for  any
three real  numbers $x,y,z$, we  have $$[x+y+z]-[x]-[y]-[z]=[\{x \}+\{y
\}+\{z\}].$$   Therefore, the  sign of  the  determinant of  $h$ is  $-
(-1)^{[\Sigma_1 ]}$  where $\Sigma_1$  is the sum $\sum  _{j=1}^3 \{\frac
{k_j}{d}  \}$. By  the condition  (\ref{S}) (see  [4] of  the Remarks
following the definition of  (\ref{S})), the integral part of $\Sigma_1$
is either $0$ or  $2$ and hence the sign of the  determinant of $h$ is
negative. \\

The  same  argument  shows  that  the determinant  of  $h_s$  is  also
negative, where $h_s$ is the skew-Hermitian form which is obtained from $h$
by   changing   $x_j=e^{2\pi   i \frac{k_j}{d}}$   to   $x_j^{(s)}=e^{2\pi   i
\frac{k_js}{d}}$. Here  $s\in (\Z/d\Z)^*$ is  viewed as an  element $(s)$ of
the  Galois   group  of   the  cyclotomic  extension   $E/\Q$.  The
determinant  of $h_s$  is  $\det (h)^{(s)}$,  and  is negative,  whence
$\det(h)$ is totally  negative. This proves the ``only  if'' part of the
second part of the Lemma. \\

The ``if'' part follows by retracing the proof of the ``only if'' part
backwards.
\end{proof}

\subsection{Relation  of the skew-Hermitian form  $h$ 
with  a  subgroup of $U(h)$}

Let ${\mathcal O}_E$, ${\mathcal O}_F$ be the ring of integers of 
$E$ and $F$. 
Suppose $\Gamma  \subset GL_2({\mathcal O}_E)$ is  the subgroup generated  by the
matrices
\[A= \begin{pmatrix}  x_1x_2 & 1-x_1 \\  0 & 1  \end{pmatrix}, \quad B
=\begin{pmatrix} 1 & 0 \\  x_2(1-x_3) & x_2x_3\end{pmatrix}. \] It can
be shown (for example, see \cite {V}, Lemmas 14 and 15 and Proposition
18), that $\Gamma$ preserves  the skew-Hermitian form $h$ of the preceding
subsection  and  that  $\Gamma  $  acts irreducibly  on  $E^2$  (the
irreducibility is implied  by the fact that the  determinant of $h$ is
non-zero,   since   it  is   a   nonzero  multiple   of
$1-x_1x_2x_3$. The number $1-x_1x_2x_3$ is nonzero since the sum $\sum
\frac{k_j}{d}$ is not an integer under the assumption (\ref{S})).

\begin{lemma} \label{Sandfinite} The group  $\Gamma $ is finite if and
only if  the condition (\ref{S}) holds for  the numbers $\frac{k_j}{d}$
$(j=1,2,3)$.
\end{lemma} 

\begin{proof}   It    is   enough   to   show,    because   of   Lemma
\ref{Sandhermitian}, that  $\Gamma$ is finite  if and only  if $h$ is
totally anisotropic. This is proved in Lemma~\ref{definite} below, for
general $n \geq 2$.
%
%If  $h$  is totally  anisotropic,  then  its  unitary group  $U_{\infty}=
%U(h)(F \otimes  _{\Q}  \R)$ is  compact,  and  since  
%$\Gamma  \subset U(h)({\mathcal O}_F)$  is a discrete  subgroup of  $U_{\infty}$, 
%it  follows that $\Gamma $ is finite. \\
%
%If  $\Gamma $  is finite,  then by  the averaging  process,  $\Gamma $
%preserves a  skew-Hermitian anisotropic form $h'$,  and by irreducibility
%of the action of  $\Gamma$, there is a unique such  form up to scalar
%multiples. Hence $h'$ is a  multiple of $h$. Therefore, $h$ is totally
%anisotropic.
\end{proof}

\subsection{The dihedral case}

Suppose  that the  finite group  $\Gamma \subset GL_2({\mathcal O}_F)$  
of  the preceding subsection has the property that  it has an abelian 
normal subgroup of index two.  We then say that $\Gamma $  is {\bf dihedral}.  
Note that $\Gamma $ is generated by two elements (namely $A,B$). \\

\begin{lemma} $\Gamma$ is dihedral if  and only if two  of the three
elements $A$, $B$, $C=AB$  have trace zero,  i.e., if and  only if two  
of the numbers $x_1x_2, x_2x_3, x_3x_1$ are equal to $-1$.
\end{lemma}

\begin{proof} Suppose $\Gamma  $ is dihedral and $N$ is an abelian normal
subgroup of index two. Since $\Gamma $ acts irreducibly on $\C ^2$, it
follows that  $\Gamma$ is not  abelian, and hence there  is an element
$g \notin  N$ in $\Gamma  $. Now  $N$ cannot  consist of  scalars. For,
otherwise the group generated by $N$ and $g$ would be abelian. \\

Let now $g \notin N$ be arbitrary. Then $g$ normalises (but  does not centralise)  the non-scalar
abelian  (and   hence  may  be   assumed  to  be   diagonal)  subgroup
$N$. Therefore, $g$ acts on $N$  by the map switching the two diagonal
entries of  an element $a\in N$. Hence  $g$ is of the  form $tw$ where
$w=\begin{pmatrix} 0 & 1 \\ 1  & 0\end{pmatrix}$ and $t$ is a diagonal
matrix; hence the element $g\notin N$ has trace zero. \\

Since  $\Gamma$  is  generated  by  any  two  of  the  three  matrices
$A,B,C=AB$, it follows  that two of these elements  cannot lie in $N$;
therefore, two of  the elements, say $A$ and $B$  have zero trace; this
means that  $x_1x_2+1=0, x_2x_3+1=0$  (a small computation  shows that
${\rm trace}(C)=(1+x_1x_3)x_2$; hence trace  $C$ being zero implies that
$x_2x_3=-1$. Thus a similar statement holds if $A,C$ do not lie in the 
subgroup $N$: $x_1x_2=x_2x_3=-1$). This proves the lemma.
\end{proof}

The lemma  means that the numbers $\frac{k_1+k_3}{d}=\frac{1}{2}$ and
$\frac{k_2+k_3}{d}=\frac{1}{2}$ (say);                        suppose
$\frac{k_1+k_2}{d}=\frac{p}{m}$, for some $p$ coprime to $m$.      
Then      it     follows     that
$\frac{k_1}{d}=\frac{k_2}{d}=\frac{p}{2m}$          and         that
$\frac{k_3}{d}=\frac{m-p}{2m}$.  This  is   the  first  part  of  Theorem
\ref{n=2}.

\subsection{Finite non-dihedral subgroups $\Gamma$}

It is well known that  any irreducible non-dihedral finite subgroup of
$PGL_2(\C)$  is  the  group  of  symmetries  of  one  of  the  platonic
solids. We however  do not use this. For the  sake of a self-contained
exposition, we will instead prove a weaker form which will suffice for
the proof of Theorem~\ref{n=2} (the  proof is adapted from Section 4,
Chapter 5, \cite{LT}).

\begin{proposition}  Suppose  $\G \subset  GL_2(\C)$  is a  finite
non-dihedral irreducible  subgroup with $Z$ the centre of $\Gamma$. 
Then the  order $m$ of any  element of the quotient 
$\G /Z$ does not exceed $5$, i.e., $m=1,2,3,4,5$. 
\end{proposition}

\begin{proof}  Consider the  action  of the  group  $GL_2(\C)$ on  the
projective line ${\mathbb P}^1(\C)\simeq  GL_2(\C)/B$ where $B$ is the
group  of  upper triangular  matrices.  This  is  the action  by  left
translation on  $GL_2(\C)/B$.  Restrict the  action to $\G$.  If $g\in
\G$ is not a scalar,  then $g$ (being diagonalisable), has exactly two
fixed  points  in ${\mathbb  P}^1(\C)$.  Moreover,  since  $\G $  acts
irreducibly on $\C^2$,  it follows that the centre  of $\G$ is exactly
the group  of scalar  matrices which  lie in $\G$.  Denote by  $g$ the
order of the quotient group $\G/Z$  and by $z$ the order of the centre
of $\G$. Then the order of $\G$ is $gz$. \\

Denote by $X$  the subset of ${\mathbb P}^1(\C)$  of points which
are fixed by some non-central  element of $\G$. Since each non-central
element of $\G$ has only two fixed points, it follows that $X$ is
finite. We claim that the first projection of the set $\Omega=\{(\gamma,x)\in \G
\times X: \gamma x=x \}$ to $\Gamma $ is surjective. If $\gamma \in \G$ is in
the centre  of $\G$,  then $\gamma$  fixes all of  the projective  line and
hence  fixes all  of $X$;  therefore, the  preimage of  $\gamma$ under
$\Omega \ra \G$  is all of $(\gamma, X)$. If  $\gamma \in \G \setminus Z$,  then it has
two fixed  points in ${\mathbb  P}^1(\C)$ both of which  by definition
lie in $\Omega$. We have therefore the equality 
\begin{equation}    
  \label{first}    
  \mathrm{Card}(\Omega) =  \mathrm{Card}(X)   z    + 2(gz-z). 
\end{equation} 
Note that $\Gamma$ acts on $X$. 
Write $X$  as a disjoint union of orbits
$\G x_i$, whose number is  $t$ say. Each isotropy $\G _{x_i}$ contains
the centre $Z$  of $\G$ and if $g_i$ denotes the  order of the quotient
group $\G_{x_i}/Z$ then
\begin{equation} 
  \label{second} 
  \mathrm{Card}(X) = \sum _{i=1}^t \mathrm{Card}(\G x_i) = \sum   \frac{g}{g_i}.
\end{equation}   
Now   consider  the   second
projection $\Omega \ra X$. The preimage of any point $x \in X$ is
$(\G_x,x)$;  the order of the isotropy of  any element of the
orbit $\G  x_i$ is the  same, namely $g_iz= \mathrm{Card}(\G _{x_i})$.  
Therefore, we get
\begin{equation} 
  \label{third}  
  \mathrm{Card}(\Omega)=\sum _{i=1}^t  \sum _{x\in \G x_i}     \mathrm{Card}     (\G _{x_i})     
  =\sum_{i=1}^t     \frac{g}{g_i} (g_iz) = gtz. 
\end{equation}  
Comparing the  last three equations,  we see that
\[ \mathrm{Card}(\Omega) = 2(gz-z)+ (\sum _{i=1}^t \frac{g}{g_i})z= gtz .\]
Dividing throughout by $z$ in the last equation, we get 
\begin{equation}    \label{classequation}   2(g-1)+    \sum   _{i=1}^t
\frac{g}{g_i}=gt.\end{equation}

We first  show that $t\leq 3$,  by using the last  equality. Note that
since $\G_i= \G_{x_i}$  is the isotropy of an  element (namely $x_i$) of
$X$  we  have $\G_i\neq  Z$  and  hence  $g_i \> (= \mathrm{Card}(\G_i/Z))  \geq
2$. Therefore,  by (\ref{classequation}), we see  that $gt\leq 2(g-1)+
t\frac{g}{2}= 2g-2+\frac{gt}{2}$.  Dividing throughout by  $g$ in this
inequality   and    rearranging   terms   we    get   $\frac{t}{2}\leq
2-\frac{2}{g}<2$, i.e., $t\leq 3$, since $t$ is an integer. \\

We  now  eliminate  the  possibility  that  $t=1,2$.  If  $t=1$,  then
(\ref{classequation})      shows      that     $2g-2+\frac{g}{g_1}=g$,
i.e.,  $g+\frac{g}{g_1}=2$. Since  $g_i$ divides  $g$ ($g_i$  being the
order of  the subgroup  $\G_i/Z$ divides the  order $g$ of  $\G/Z$), it
follows that $g=1$ and $g=g_i$. But $g=1$ means that $\G/Z$ is trivial,
i.e.,  $\G$ is  central and  therefore not  an irreducible  subgroup of
$GL_2(\C)$.   Hence  $t\neq   1$.  If   $t=2$,  then   again  equation
(\ref{classequation})   shows   that  $2=\frac{g}{g_1}+\frac{g}{g_2}$,
which means that $g=g_1=g_2$ and hence $\G= \G_1$ $(= \G_2)$ and is therefore
an  abelian  group;  hence  $\G$  cannot  be  irreducible.  Therefore,
$t=3$. \\

>From (\ref{classequation}) we  now get (after dividing by  $g$ on both
sides)
\begin{equation}                                 
  \label{orbitequation}
  1+\frac{2}{g}=\frac{1}{g_1}+\frac{1}{g_2}+\frac{1}{g_3}.
\end{equation}
Assume,  as we  may, that  $g_1 \leq g_2  \leq g_3$.  Then  the equality
(\ref{orbitequation})  shows that  $1<3\frac{1}{g_1}$,  i.e., $g_1  <3$;
since $g_1\geq 2$, it follows that $g_1=2$.

We      again       get      from      (\ref{orbitequation})      that
$1+\frac{2}{g}=\frac{1}{2}+\frac{1}{g_2}+\frac{1}{g_3}$,  with  $2\leq
g_2\leq   g_3$.  Therefore,  $1<\frac{1}{2}+\frac{2}{g_2}$;   that  is,
$g_2<4$.   Therefore,  $g_2=2,3$.   If  $g_2=2$,   then   equation
(\ref{orbitequation})                    shows                    that
\[1+\frac{2}{g}=\frac{1}{2}+\frac{1}{2}+\frac{1}{g_3}.\]       
Therefore,
$g=2g_3$; in other words, $\G_3$ is an abelian subgroup of index $2$ in
$\G$, which means that $\G$ is dihedral, contradicting the assumptions
of the proposition. Therefore, $g_2=3$ is the only possibility. \\

Now   (\ref{orbitequation})  shows   that  $1+\frac{2}{g}=\frac{1}{2}+
\frac{1}{3}+\frac{1}{g_3}$,    with    $g_3\geq    3$.   Hence,    $1<
\frac{5}{6}+\frac{1}{g_3}$,  which yields  $g_3=3,4,5$. Hence  we have
proved that every  non-central element of $\G$ lies  in a conjugate of
one of  the subgroups $G_1$, $G_2$  and $G_3$ which  have orders $2,3$
and $g_3=3,4,5$ respectively. This proves the proposition.
\end{proof}

We  now return to  the situation  of Lemma  \ref{Sandfinite}. Consider
$x_j=e^{2\pi  i \frac{k_j}{d}}$  ($j=1,2,3$).  The irreducible  finite
subgroup $\G$ is generated by $A=\begin{pmatrix} x_1x_2 & 1-x_1 \\ 0 &
1  \end{pmatrix}$  and  $B=\begin{pmatrix}  1  &  0  \\  x_2(1-x_3)  &
x_2x_3  \end{pmatrix}$.  The   image  of  $\G$  in  $PGL_2(\C)$
contains the images $A'$ and $B'$ of $A$ and $B$ respectively; clearly
the orders of  $A'$ and $B'$ are respectively the  orders of the roots
of  unity $x_1x_2=e^{2\pi i  \frac{k_1+k_2}{d}}$ and  $x_2x_3=e^{2\pi i
\frac{k_2+k_3}{d}}$. \\

A   computation  shows  that the matrix
$$C = AB = \begin{pmatrix}  x_2(1-x_3+x_1x_3)  &  x_2x_3(1-x_1) \\  x_2(1-x_3)  &
x_2x_3\end{pmatrix}$$ 
has eigenvalues  $x_2$ and  $x_1x_2x_3$; clearly
the  order  of  the image  of  $C$  in  $PGL_2(\C)$  is the  ratio  of
these eigenvalues       $\frac{x_1x_2x_3}{x_2}= x_3x_1=e^{2\pi       i
\frac{k_3+k_1}{d}}$.  From the proposition follows the

\begin{corollary}  \label{mucorollary}   If  condition  (\ref{S})
holds, and $\frac{k_1}{d},\frac{k_2}{d},  \frac{k_3}{d}$ is not in the
dihedral  case,  then  the  fractions $\mu  _1=\frac{k_2+k_3}{d},  \mu
_2=\frac{k_3+k_1}{d}, \mu _3=\frac{k_1+k_2}{d}$  are in the finite set
$S$   of  fractions  of   the  form   $\frac{t}{u}$  with   $t<u$  and
$u=1,2,3,4,5$.
\end{corollary}

\subsection{A finite list}

Since the set $S$  in   Corollary    \ref{mucorollary} is finite, 
clearly the set of fractions $\frac{k_1}{d},\frac{k_2}{d} ,\frac{k_3}{d}$ obtained
from the set of $\mu _1, \mu _2, \mu _3$ in $S$ is also finite. 
Working up to permutation and up to the equivalence defined before, we may check that if 
$\frac{k_1}{d}, \frac{k_2}{d}, \frac{k_3}{d}$ further satisfies the condition (\ref{S}), 
then the corresponding $(\lambda,\mu,\nu)$  lie in the finite list in Table 1 
below (we discard the dihedral cases with $1 \leq p < m \leq 5$). 
This implies Theorem \ref{n=2}. \qed

\vspace{.2in}
\setlongtables
\begin{longtable}{| c | c c c | c c c | c c c | c |}
\caption[Schwarz's List]{(Non-dihedral) Schwarz's List}\\
\toprule
%\begin{tabular}{| c | c c c | c c c | c c c | c |}
%\hline
$d$  &  $\mu_1$  & $\mu_2$  & $\mu_3$ & $k_1/d$ & $k_2/d$ & $k_3/d$  &  $\lambda$  & $\mu$  & $\nu$  & Wiki-row \\[.1mm]
\midrule
\endfirsthead
\toprule
$d$  &  $\mu_1$  & $\mu_2$  & $\mu_3$ & $k_1/d$ & $k_2/d$ & $k_3/d$  &  $\lambda$  & $\mu$  & $\nu$  & Wiki-row \\[.1mm]
\midrule 
\endhead
\midrule
\endfoot
\bottomrule
\endlastfoot

12 &  2/3  & 2/3  & 1/2 &  1/4  &  1/4  & 5/12  &   1/2    &  1/3   &  1/3   &  2 \\
6  &  2/3  & 2/3  & 1/3 &  1/6  & 1/6   & 1/2   &   2/3    &  1/3   &  1/3   &  3 \\
30 &  2/3  & 2/3  & 3/5 & 3/10  & 3/10  & 11/30 &   2/5    &  1/3   &  1/3   &  7 \\
60 &  2/3  & 3/5  & 1/2 & 13/60 & 17/60 & 23/60 &   1/2    &  2/5   &  1/3   & 14 \\
30 &  2/3  & 3/5  & 2/5 & 1/6   & 7/30  & 13/30 &   3/5    &  2/5   &  1/3   & 15 \\
24 &  3/4  & 2/3  & 1/2 & 5/24  & 7/24  & 11/24 &   1/2    &  1/3   &  1/4   &  4 \\
12 &  3/4  & 3/4  & 1/3 & 1/6   & 1/6   & 7/12  &   2/3    &  1/4   &  1/4   &  5 \\
10 &  3/5  & 3/5  & 3/5 & 3/10  & 3/10  & 3/10  &   2/5    &  2/5   &  2/5   & 11 \\
60 &  4/5  & 2/3  & 1/2 & 11/60 & 19/60 & 29/60 &   1/2    &  1/3   &  1/5   &  6 \\
30 &  4/5  & 2/3  & 1/3 & 1/10  & 7/30  & 17/30 &   2/3    &  1/3   &  1/5   & 12 \\
15 &  4/5  & 2/3  & 2/5 & 2/15  & 4/15  & 8/15  &   3/5    &  1/3   &  1/5   & 10 \\
20 &  4/5  & 3/5  & 1/2 & 3/20  & 7/20  & 9/20  &   1/2    &  2/5   &  1/5   & 9  \\
30 &  4/5  & 4/5  & 1/3 & 1/6   & 1/6   & 19/30 &   2/3    &  1/5   &  1/5   & 8  \\
10 &  4/5  & 4/5  & 1/5 & 1/10  & 1/10  & 7/10  &   4/5    &  1/5   &  1/5   & 13 \\
\end{longtable}

All but the last column of Table 1 was generated using Pari-gp. The table 
is (the non-dihedral part of) Schwarz's well-known 1873 list \cite{Sch}, see 
Wikipedia: {\tt https://en.wikipedia.org/wiki/Schwarz's\_list}. 
The last column of Table 1 contains the row number of the corresponding entry 
in the Wikipedia table.  
Each of the 15 rows in that table is hit (the 1st row  being the dihedral case).

\section{The case $n \geq 3$}
\label{ngeq3}

We now use a bootstrapping argument to prove the remaining parts of 
Theorem~\ref{numbertheoretic}.  We start with the following obvious lemma.

\begin{lemma}
  \label{bootstrap}
  Let $n \geq 3$. Assume that the g.c.d. of $d, k_1, k_2, \cdots, k_{n+1}$ is equal to 1. If 
  these integers satisfy the condition (\ref{SS}), then  
  so do $d, l_1, l_2, \cdots, l_{m+1}$, for all subsets  
  $\{l_i\}$  of cardinality $m+1$ of the $\{k_j\}$, for $2 \leq m \leq n$.
\end{lemma}

We remark that however the g.c.d. of $d, l_1, l_2, \cdots, l_{m+1}$ may no longer
necessarily be equal to 1. 
 
\subsection{$n=3$} We use the lemma to treat the case $n = 3$ using the result 
for $n = 2$ proved in Theorem~\ref{n=2}. 
A complete list of (non-dihedral) tuples $(d, l_1, l_2, l_3)$ with the g.c.d. of
$d, l_1, l_2, l_3$ equal to 1 and satisfying condition (\ref{S}) 
(let us call the corresponding triplet $(l_1, l_2, l_3)$ primitive) is easily 
generated from Table 1, and is provided in Table 2 below.

\vspace{.2in}
\setlongtables
\begin{longtable}{| c | c |}
\caption[Schwarz's List]{Primitive (non-dihedral) Schwarz triplets}\\
\toprule
%\begin{tabular}{| c | c c c | c c c | c c c | c |}
%\hline
$d$  &  $(l_1, l_2, l_3)$ with $l_1 \leq l_2 \leq l_3$  \\[.1mm]
\midrule
\endfirsthead
\toprule
$d$  &  $(l_1, l_2, l_3)$ with $l_1 \leq l_2 \leq l_3$ \\[.1mm]
\midrule 
\endhead
\midrule
\endfoot
\bottomrule
\endlastfoot

6  & {\small $(1,1,1)$, $(5,5,5)$} \\
   & {\small $(1,1,3)$, $(3,5,5)$} \\ \hline
10 & {\small $(1,1,1)$, $(3,3,3)$, $(7,7,7)$, $(9,9,9)$} \\
   & {\small $(1,3,3)$, $(3,9,9)$, $(1,1,7)$, $(7,7,9)$} \\ \hline
12 & {\small $(1,3,5)$, $(7,9,11)$ $\>$ (each with multiplicity 2)} \\
   & {\small $(1,2,7)$, $(5,10,11)$ $\>$ (each with multiplicity 2)} \\
   & {\small $(1,2,2)$, $(5,10,10)$,  $(2,2,7)$, $(10,10,11)$} \\ 
   & {\small $(1,3,3)$, $(3,3,5)$, $(7,9,9)$, $(9,9,11)$} \\ \hline
15 & {\tiny $(1,2,4)$, $(2,4,8)$, $(1,4,8)$, $(7,13,14)$, $(1,2,8)$, $(7,11,14)$, $(7, 11, 13)$, $(11,13,14)$} \\ \hline
20 & {\tiny $(1,3,7)$, $(1,3,9)$, $(1,7,9)$, $(3,7,9)$, $(11,13,17)$, $(11,13,19)$, $(11,17,19)$, $(13,17,19)$} \\ \hline
24 & {\tiny $(1,5,7)$, $(1,5,11)$, $(1,7,11)$, $(5,7,11)$, $(13,17,19)$, $(13,17,23)$, $(13,19,23)$, $(17,19,23)$} \\ \hline
30 & {\tiny $(1,5,5)$, $(5,5,7)$,  $(11,25,25)$, $(5,5,13)$, $(17,25,25)$, $(5,5,19)$, $(23,25,25)$, $(25,25,29)$} \\
   & {\tiny $(3,7,17)$, $(19,21,29)$, $(1,9,11)$, $(13,23,27)$ $\>$ (each with multiplicity 2)} \\ 
   & {\tiny $(1,9,9)$, $(3,3,7)$,  $(9,9,11)$, $(13,27,27)$, $(3,3,17)$, $(19,21,21)$, $(23,27,27)$, $(21,21,29)$} \\
   & {\tiny $(5,7,13)$, $(1,5,19)$,  $(17,23,25)$, $(11,25,29)$ $\>$ (each with multiplicity 2)} \\ \hline
60 & {\tiny $(1,11,19)$, $(7,13,17)$,  $(1,11,29)$, $(7,13,23)$, $(7,17,23)$, $(1,19,29)$, $(13,17,23)$, $(11,19,29)$, } \\
   & {\tiny $(31,41,49)$, $(37,43,47)$,  $(31,41,59)$, $(37,43,53)$, $(37,47,53)$, $(31,49,59)$, $(43,47,53)$, $(41,49,59)$} \\
\end{longtable}
 
We remark that each line in Table 2 represents one $(\Z/d\Z)^*$-orbit, and so has cardinality $\varphi(d)$, 
though for $d = 12, 30$ some triplets occur with multiplicity 2, and for $d = 60$ both lines form one orbit.  

Now suppose that $(d, k_1, k_2, k_3, k_4)$ satisfies the condition (\ref{SS}). %with $k_1 \leq k_2 \leq k_3 \leq k_4$. 
Assume that this tuple is primitive, i.e., the g.c.d. of $d,k_1,k_2,k_3,k_4$ is equal to 1.
Then, by the lemma, every sub-tuple $(d, l'_1, l'_2, l'_3)$, where the $l'_i$ are obtained by discarding
one $k_j$, also satisfies the condition (\ref{S}). Note that the g.c.d. of $d, l'_1, l'_2, l'_3$ does not have to be 1. 
By Theorem~\ref{n=2}, we see that $(d, l'_1, l'_2, l'_3)$ is a positive integral multiple of 
some $(d_1, l_1, l_2, l_3)$ occurring in Table 2 (or is dihedral), up to permutation.

Ignoring (momentarily) the tuples containing multiples of a dihedral Schwarz triplet, we see that $d$ must be 
bounded by 120. Indeed, if say,  
\begin{eqnarray}
  \label{smallerlevel}
  \begin{gathered}
    (d,k_1,k_2,k_3)  =   a \cdot (d_1, l_1, l_2, l_3) \\ 
    (d,k_1,k_2,k_4)  =   b \cdot (d_2, m_1, m_2, m_4),
  \end{gathered}
\end{eqnarray}
for some tuples $(d_1, l_1, l_2, l_3)$ and $(d_2, m_1, m_2, m_4)$ in Table 2 (up to permutation), and
for some positive integers $a$, $b$,  then by primitivity, the g.c.d. of $a,b$ has to be equal to 1, hence
$d = ad_1 = bd_2$, so $a|d_2$, so $d | d_1d_2$, so $d$ divides the l.c.m. of all $d$ occurring in Table 2, 
which is 120. \\

This reduces the problem of checking which primitive quadruples $(d, k_1, k_2, k_3, k_4)$ satisfy condition (\ref{SS}) 
(for $n = 3$) to a finite check. Table 3 lists all such primitive tuples which satisfy the 
property that every sub-tuple obtained by dropping exactly one of the $k_j$ arises from Table 2, by possibly scaling up from a 
smaller denominator. (The fact that, for instance, the g.c.d. of $a,b$ is 1 greatly reduces the number of 
smaller denominators that one has to consider.)
\vspace{.2in}
\setlongtables
\begin{longtable}{| c | c |}
\caption[Schwarz's List]{Possible (non-dihedral) Schwarz 4-tuplets}\\
\toprule
%\begin{tabular}{| c | c c c | c c c | c c c | c |}
%\hline
$d$  &  $(k_1, k_2, k_3, k_4)$ with $k_1 \leq k_2 \leq k_3 \leq k_4$  \\[.1mm]
\midrule
\endfirsthead
\toprule
$d$  &  $(k_1, k_2, k_3, k_4)$ with $k_1 \leq k_2 \leq k_3 \leq k_4$  \\[.1mm]
\midrule 
\endhead
\midrule
\endfoot
\bottomrule
\endlastfoot

6  & {\small $(1,1,1,1)$, $(5,5,5,5)$} \\
   & {\small $(1,1,1,3)$, $(3,5,5,5)$} \\ \hline
10 & {\small $(1,1,1,1)$, $(3,3,3,3)$, $(7,7,7,7)$, $(9,9,9,9)$} \\
   & {\small $(1,3,3,3)$, $(3,9,9,9)$, $(1,1,1,7)$, $(7,7,7,9)$} \\ \hline
12 & {\small $(1,2,2,7)$, $(5,10,10,11)$ $\>$ (each with multiplicity 2)} \\  
   & {\small $(1,3,3,5)$, $(7,9,9,11)$ $\>$ (each with multiplicity 2)} \\  
   & {\small $(1,2,2,2)$, $(5,10,10,10)$,  $(2,2,2,7)$, $(10,10,10,11)$} \\  \hline
15 & {\small $(1,2,4,8)$, $(7,11,13,14)$ $\>$ (each with multiplicity 4)} \\ \hline
20 & {\small $(1,3,7,9)$, $(11,13,17,19)$ $\>$ (each with multiplicity 4)} \\ \hline
24 & {\small $(1,5,7,11)$, $(13,17,19,23)$ $\>$ (each with multiplicity 4)} \\ \hline
30 & {\tiny $(1,5,5,5)$, $(5,5,5,7)$,  $(11,25,25,25)$, $(5,5,5,13)$, $(17,25,25,25)$, $(5,5,5,19)$, $(23,25,25,25)$, $(25,25,25,29)$} \\
   & {\tiny $(1,9,9,11)$, $(3,3,7,17)$, $(19,21,21,29)$, $(13,23,27,27)$ $\>$ (each with multiplicity 2)} \\ 
   & {\tiny $(1,9,9,9)$, $(3,3,3,7)$,  $(9,9,9,11)$, $(13,27,27,27)$, $(3,3,3,17)$, $(19,21,21,21)$, $(23,27,27,27)$, $(21,21,21,29)$} \\
   & {\tiny $(1,5,5,19)$, $(5,5,7,13)$,  $(11,25,25,29)$, $(17,23,25,25)$ $\>$ (each with multiplicity 2)} \\ \hline
60 & {\tiny $(1,11,19,29)$, $(7,13,17,23)$, $(31,41,49,59)$, $(37,43,47,53)$ $\>$ (each with multiplicity 4)} \\ \hline
120 & {\small No tuplets} \\
\end{longtable}

It is now straightforward to check that of these tuples, exactly two, namely $(1, 1, 1, 1)$ and $(5,5,5,5)$, both for $d = 6$, 
satisfy condition (\ref{SS}). Since these tuples are equivalent, we have proved one half of 
Theorem~\ref{numbertheoretic} (ii) (for $n = 3$).  \\

To treat the other half, we now assume that at least one of the sub-tuples of $(d, k_1, k_2, k_3, k_4)$ 
is a multiple of a dihedral triplet. By rearranging the $k_i$, we may assume that the first sub-tuple in (\ref{smallerlevel}), 
is dihedral of the form:
\begin{eqnarray*}
  (d, k_1, k_2, k_3) = a \cdot (d_1, l_1, l_2, l_3) = a \cdot (2m, p, p, m-p),
\end{eqnarray*}
for some $1 \leq p < m$, with the g.c.d. of $p,m$ equal to $1$. Clearly the second tuple in (\ref{smallerlevel})
cannot be dihedral, for if
\begin{eqnarray*}
  (d, k_1, k_2, k_4) = b \cdot (d_2, m_1, m_2, m_4) = b \cdot (2l, q, q, l-q),
\end{eqnarray*}
for $1 \leq q < l$, with the g.c.d. of $q,l$ equal to $1$, then $k_1 + k_3 = d/2 = k_1 + k_4$
so that $k_4 = k_3 = a(m-p)$. Then $k_2/d + k_3/d = (ap/2ma) + a(m-p)/2ma = 1/2$, and
similarly $k_1/d + k_4/d = 1/2$ so that condition (\ref{SS}) fails. (A similar argument applies if  
the second tuple $(d, k_1, k_2, k_4)$ equals $b \cdot (2l, q, l- q, q)$ instead.)  \\

This means  that the second tuple above is a multiple of a non-dihedral tuple  
$(d_2, m_1, m_2, m_4)$, occurring in the finite  list in Table 2, up to 
permutation.  
As before, we have $d = 2ma = bd_2$, and since the g.c.d. of $a,b$ is 1, we have $a | d_2$ (and
so is bounded) and $b | 2m$.  Moreover $ap = k_1 = bm_1$ implies that $b | p$, and since the
g.c.d. of $p,m$ is 1, we see that $b = 1,2$. However, the latter case cannot occur: if
$b = 2$ is even, then $a$ is odd and $p$ is even, so $d_2 = 2ma / b = ma$ is odd (else
$m$ is even, so 2 divides the g.c.d. of $d, k_1, k_2, k_3, k_4$, contradicting primitivity). 
But then $d_2$ must equal the only odd entry $15$ in Table 2, which is impossible, 
since $k_1 = ap = k_2$, but all triplets for $d = 15$ have distinct entries. \\ 

Thus $b = 1$ and  
$(d, k_1, k_2, k_3, k_4)$ has the shape:
\begin{eqnarray}
  \label{dihedral+finite}
  (2ma = d_2, ap=m_1, ap=m_2, a(m-p), m_4).
\end{eqnarray}
Since  $d_2$ is bounded, both $a$ and $m$ divide $d_2/2$, and $1 \leq p < m$ (with $p$ is coprime to $m$), 
clearly there are only finitely many possibilities for such tuples. Moreover, since 
$m_1 = m_2$, an inspection of Table 2 shows that 
$d_2$ can only be one of $6,10,12,30$.  \\

The following table lists all possibilities for tuples having 
shape (\ref{dihedral+finite}) above.

\vspace{.2in}
\setlongtables
\begin{longtable}{| c | c |}
\caption[Schwarz's List]{Possible tuplets of the form (\ref{dihedral+finite})}\\
\toprule
%\begin{tabular}{| c | c c c | c c c | c c c | c |}
%\hline
$d$  &  $(k_1, k_2, k_3, k_4)$   \\[.1mm]
\midrule
\endfirsthead
\toprule
$d$  &  $(k_1, k_2, k_3, k_4)$   \\[.1mm]
\midrule 
\endhead
\midrule
\endfoot
\bottomrule
\endlastfoot

6  & {\small $(1,1,2,1)$} \\
   & {\small $(1,1,2,3)$} \\ \hline
10 & {\small $(1,1,4,1)$} \\
   & {\small $(1,1,4,7)$} \\
   & {\small $(3,3,2,1)$} \\ 
   & {\small $(3,3,2,3)$} \\ \hline
12 & {\small $(2,2,4,1)$, $(2,2,4,7)$} \\ 
   & {\small $(3,3,3,1)$, $(3,3,3,5)$} \\ \hline
30 & {\small $(3,3,17,7)$, $(3,3,12,17)$,  $(9,29,6,1)$, $(9,9,6,11)$} \\
   & {\small $(5,5,10,1)$, $(5,5,10,7)$,  $(5,5,10,13)$, $(5,5,10,19)$} \\ 
\end{longtable}

A quick inspection now shows that of these possibilities only the tuple $(1,1,2,1)$ for $d = 6$
satisfies (\ref{SS}), proving the second half of Theorem~\ref{numbertheoretic} (ii). This 
completes the proof of the case $n = 3$. 

\subsection{$n=4$} This follows easily from Lemma~\ref{bootstrap} and the just established case
$n = 3$. Indeed by the lemma, the only possible $5$-tuplets of $k_i$ would occur for $d = 6$
and would be $(1,1,1,1,1)$ or $(1,1,1,1,2)$ up to equivalence. But only
the former satisfies the condition (\ref{SS}). 

\subsection{$n = 5$} The same argument shows that the only possible $6$-tuplet
of $k_i$ is $(1,1,1,1,1,1)$ for $d = 6$. But one easily checks that this tuple
fails to satisfy the condition (\ref{SS}), so there are no $6$-tuples satisfying (\ref{SS}).

\subsection{$n\geq6$} Finally, Lemma~\ref{bootstrap} shows that there are no tuplets for $n \geq 6$ satisfying
the condition (\ref{SS}) since we have just shown there is none for $n = 5$. This completes the proof of 
Theorem~\ref{numbertheoretic}. \qed.

\section{On finiteness of some monodromy}
\label{monodromy}

Consider the  space $S=\{(z_1,\cdots,  z_{n+1})\in \C ^{n+1}:  z_i\neq z_j,
 \forall  i\neq j\}$, for $n + 1 \geq 3$. Let $d\geq  2$ be an  integer; fix integers
$k_1,k_2, \cdots,  k_{n+1}$ with $1\leq  k_i \leq d-1$ such  that $d \Z + \sum
k_i\Z =\Z$. The space of solutions $(x,y)$ to the equation
\[y^d = (x-z_1)^{k_1}(x-z_2)^{k_2}\cdots  (x-z_{n+1})^{k_{n+1}} \]
is the  affine part  of a smooth  projective curve  $C=C_{d,k_i}$. Write $\mu _i=\frac{k_i}{d}$; our assumptions imply that $\mu _i \in \Q \setminus \Z$. 
Write, as in \cite{Del-Mos}, \cite{Coh-Wol},  $\mu _{\infty}=2-\sum _{i=1}^{n+1} \mu _i$; the numbers $\mu _i$ and $\mu _{\infty}$ record the ramifications  
at the $z_i$ and at $\infty$; we assume that $\mu _{\infty}$ is also not integral so that the curve $C$  is ramified at infinity as well. \\

The group $G=\Z/d\Z$  acts on the curve $C$; the  action on the affine
part is given  by $y\mapsto \omega y$ where $\omega$  is a $d$-th root
of unity. Consequently, the group $G$ operates on the first cohomology
of  the curve  $C$ with  rational  coefficients; denote  by $M_d$  the
direct sum of the cohomology over  $\C$, on which a fixed generator of
the group $G$ operate by {\it some} primitive $d$-th root of unity. \\

The  fundamental group of  the space  $S$ acts  (by monodromy)  on the
space $M_d$.  It is well known that this fundamental group is the same
as the pure braid group  $P_{n+1}$. We classify the integers $d,n$ and
the numbers  $k_i$ for which the  image of the  fundamental group (the
monodromy group) in $\mathrm{Aut}(M_d)$ is {\it finite}.  This problem
of   finiteness  has   already  been   resolved  by   several  authors
(\cite{Ssk},  \cite{Coh-Wol},   \cite{Bod},  \cite{Sch}),  since  this
monodromy is  the same as  the monodromy of  certain Appell-Lauricella
hypergeometric functions.   However, we believe  our point of  view is
different: the explicit  description of the monodromy in  terms of the
Gassner representation  makes the proofs completely  algebraic, and is
formulated in terms of the definiteness of an explicit Hermitian form.

\begin{theorem} \label{numbertheoreticmonodromy} Suppose $n,d,k_i$ are
as in the  preceding.  Then the image of  the monodromy representation
in Aut$(M_d)$  is finite  if and only  if condition  (\ref{SS}) holds.
Thus the monodromy on $M_d$ is by a finite group if and only if
\[n\leq  4.\]  Moreover,  up  to equivalence,  the  numbers  $n,d,k_i$
satisfy the conditions given below. \\

\begin{enumerate}
\item[(i)] If $n=4$, then $d=6$ and $k_i=1$, for all $i \leq 5$. \\

\item[(ii)] If $n=3$, then $d=6$ and $k_1=k_2=k_3=1$ and 
$k_4=1 \> {\rm or} \> 2$. \\

\item[(iii)] 
If  $n=2$, then $d=2m$ and $k_1=k_2=p$  and $k_3=m-p$, or
else $d,k_i$ lie in a finite list, with $d\leq 60$.
\end{enumerate}

\end{theorem}

\begin{remark} Again,  note that the condition (\ref{SS})  is a purely
number theoretic condition; the proof of the theorem, however, depends
on an analysis of certain finite subgroups of unitary groups generated
by reflections.
\end{remark}

We   will  prove  Theorem~\ref{numbertheoreticmonodromy}   after  some
preliminaries  on  Hermitian forms  and  unitary  groups generated  by
reflections.

\section{Skew-Hermitian Forms}
\label{skewhermitian}

\begin{lemma} 
\label{anisotropic}
Suppose  $h$ is a  skew-Hermitian form on  $\C^n$.  Then
$h$ is anisotropic if and only  if the principal minors $u_i$ have the
property: $\frac{u_{j+1}u_{j-1}}{u_j^2}$ is positive, for all $j\geq 1$
(by convention $u_0=1$).
\end{lemma}

\begin{proof}  Suppose  that $h$  does  not  represent  a zero;  hence
$a_{11}\neq 0$, where  $(a_{ij})$ is the matrix of  $h$ in the standard
basis. By  the Gram-Schmidt process, there exists  an upper triangular
unipotent matrix $u\in  GL_n(\C)$ such that $^t(\overline{u})hu=h'$ is
diagonal, with diagonal entries  $\lambda _1, \cdots, \lambda _n$ say.
Now  $h$  is  anisotropic  if  and  only if  the  equivalent  $h'$  is
anisotropic. The latter  is anisotropic if and only  if the successive
ratios $\beta _j=\lambda _{j+1}/\lambda _j$ are all positive. \\

Since $u$  is a unipotent  upper triangular  matrix, the
principal   minors  of  $h$   and  $h'$   are  the   same.  Therefore,
$u_{j+1}=\lambda   _1\cdots   \lambda   _{j+1}$.   Consequently,   
$h'$  is anisotropic  if and only if  for all $j$, $$\beta
_j =  (u_{j+1}/u_j)/ (u_j/u_{j-1})$$ is positive.  
This is  equivalent to
$\beta _j= \frac{u_{j+1}u_{j-1}}{u_j^2}$  being positive,  for  all
$j$. Hence the lemma.
\end{proof}

Now suppose  that $E/F$ is a  CM extension of number fields 
and that $h$ is a skew-Hermitian form on $E^n$. Suppose that $h$
does not represent a  zero. Then the Gram-Schmidt process diagonalises
$h$.  Suppose the diagonal  entries are  $\lambda _1,  \cdots, \lambda_n$. 
Then we have
  \[\lambda _1\cdots \lambda _j=\det (h_j),\]
where $\det(h_j)$ is the principal $j\times j$ minor of $h$. Hence 
\[\lambda _j= \det (h_j)/\det (h_{j-1}).\] Write 
\[\b _j= \frac{\det (h_{j+1}) \det(h_{j-1})}{\det(h_j)^2}.\]
>From the previous lemma we obtain:

\begin{lemma}  \label{betadef} Suppose  $F\ra \R$  is an  embedding and
$E\otimes  _F\R=\C$.  Then the  skew-Hermitian  form  $h$ is anisotropic  in  
this embedding if and only if

\[\b _j >0, \quad \forall j. \]
\end{lemma}

\begin{lemma} \label{definite} Suppose $h$ is a skew-Hermitian form in
$n$   variables  over   a  CM   field  $E/F$,   and   $\Gamma  \subset
U(h)({\mathcal   O}_F)$   a  subgroup   which   acts  irreducibly   on
$\C^n$.  Then $\Gamma  $  is finite  if  and only  if  $h$ is  totally
anisotropic.
\end{lemma}

\begin{proof}  Suppose  $\Gamma$ is  finite.  Fix  {\it any}  positive
definite  Hermitian  form $H$  on  $\C^n$.  Being  a sum  of  positive
definite forms,  the average $H'(x,y)=\sum _{\g \in  \G} H(\g x,\g y)$ is
also positive  definite and is $\Gamma$-invariant. Hence  $iH'$ is a
$\G$-invariant anisotropic skew-Hermitian form on $\C^n$. 
The  irreducibility of  the action of  $\Gamma$ implies, by Schur's lemma,   
that the invariant anisotropic skew-Hermitian form $iH'$ is  a 
scalar multiple of the form $h$, for any embedding  of $F$ into $\R$.  
Hence $h$  is anisotropic over all embeddings of the field $F$. \\

Conversely, if $h$ is anisotropic at  all real places $v$ of $F$, then
$ih$ is  definite, for all  $v$, and hence the  group $U(h)(F_v)\simeq
U(ih)(F_v)$ is compact, for  all $v$.  Since $U(h)({\mathcal O}_F)$ is
a  discrete subgroup  of $U(h)(F\otimes  _{\Q} \R)$,  it  follows that
$U(h)({\mathcal O}_F)$ is finite, hence $\Gamma$ is also finite.
\end{proof}

\begin{remark}
Let $G = U(h)$. A corollary of the proof is that if a finite subgroup of 
$G({\mathcal O}_F)$ acts irreducibly on $\C^n$, then $G({\mathcal O}_F)$ 
is finite.
\end{remark}

\begin{notation}   Denote   by  $R$   the   Laurent  polynomial   ring
$\Z[X_1^{\pm  1}, \cdots,  X_{n+1}^{\pm 1}]$  in $n+1$  variables  with 
$\Z$-coefficients.  
The  map $X_i\mapsto X_i^{-1}$,  for all $i$,  induces an
involution of order two on the  ring $R$. Denote by $R^n$ the standard
free $R$  module of rank $n$  with standard basis $\e  _i$, for $ 1\leq
i\leq n$.  Define the skew-Hermitian form $h=(h_{ij})_{1\leq i,j\leq
n}$ by the formulae $h(\e_i,\e_j)=0 $ if $\mid i-j \mid \geq 2$ and
\[h(\e_i,\e_i)=\frac{1-X_iX_{i+1}}{(1-X_i)(1-X_{i+1})},             \quad
h(\e_i,\e_{i+1})= - \frac{1}{1-X_{i+1}}.\] 

\end{notation}

\begin{lemma} \label{determinant}  The form  $h$ does not  represent a
zero in $R^n$. Moreover, for each $j$, the principal $j\times j$ minor
is given by
%\[\det (h_j)
\[u_j = \frac{1-X_1\cdots X_{j+1}}{(1-X_1)\cdots (1-X_{j+1})}.\]
\end{lemma}
 \begin{proof}

In \cite{V}, the determinant of $h$ was computed to be 
\[\frac{1-X_1X_2\cdots  X_{n+1}}{(1-X_1)\cdots  (1-X_{n+1})}.\] Taking
$n=j$,  we get  the  formula for  the  determinant of  the $j\times  j$
principal  minor. Take $X_j=e^{2\pi  i \th  _j}$ to  be transcendental
with $\th_j \in  \R$  positive and close  to  $0$.  Put  $\Sigma  _{j-1}=
\th_1+\cdots + \th _{j}$  (the sum of the first $j$  terms).  
Using the equality
\begin{eqnarray}
  \label{sin}
  1-e^{2\pi i \theta}= e^{\pi i \theta} ((-2i) \sin (\pi \theta)),
\end{eqnarray}
and that $\sin(\th)$ is close to $\th$, for $\th$ small (and positive), 
we see that the numbers \[\b  _j % = \frac{\det(h_{j+1}) \det(h_{j-1})}{\det (h_j)^2}  
                               =  \frac{u_{j+1} u_{j-1}}{u_j^2} = \frac{\sin (\pi
\Sigma _{j+1}) \sin  (\pi \Sigma _{j-1}) \sin (\pi \th  _{j+1})} {\sin (\pi
\Sigma  _j)^2   \sin  (\pi  \th   _{j+2})}\]  are positive, for all $j$.   By  Lemma
\ref{anisotropic} it follows that $h$ is anisotropic.
\end{proof}

The map  $X_i\mapsto t_i= e^{2\pi \sqrt{-1}  \frac{k_is}{d}}$ maps $R$
onto the  ring ${\mathcal O}_E$  of integers in the  $d$-th cyclotomic
extension $E=\Q(e^{2\pi i /d})$.  Let $F=\Q(\cos(\frac{2\pi }{d}))$ be
the  maximal totally  real  subfield of  $E$.   We then  get a  
skew-Hermitian form on $E^n$ induced from $h$.  \\

Define, for each $j\leq n-1$, the numbers 
\[\nu   _j=  \nu   _j(s)  =   \left\{\sum   _{i=1}^{j}  \frac{k_is}{d}
\right\}.\]  Denote  by  (\ref{*})  the conditions  satisfied  by  the
numbers $n,d,k_i$:
\begin{equation}
\label{*} 
  \begin{gathered}  \e_j=  \e_j(s)\stackrel{def}{=}  (-1)^{[\nu_j(s)+\mu
_{j+1}(s)+\mu   _{j+2}(s)]}=1,  \\   \forall  s\in   (\Z/d\Z)^*  \quad
\text{and} \quad \forall j \quad with \quad 1\leq j\leq n-1.
  \end{gathered}
\end{equation}

\begin{lemma}      \label{finite}      The     group      $G({\mathcal
O}_F)=U(h)({\mathcal  O}_F)$  is finite  if  and  only if  $(n,d,k_i)$
satisfy condition (\ref{*}).
\end{lemma}

\begin{proof}  Consider the  ``Gassner  representation'' $G(X):P_{n+1}
\ra U(h, R)$  \cite{V} (recall that the ring  $R$ is the Laurent
polynomial ring in  the variables $X_j:1\leq j \leq  n+1$ with integer
coefficients).      Specializing    $X_j$     to     $x_j=e^{2\pi    i
\frac{k_j}{d}}=e^{2\pi i \mu _j}$ we obtain a representation $\rho _d$
of the pure braid group $P_{n+1}$. The image of $\rho _d$ is contained
in  the  group $G({\mathcal  O}_F)$  (e.g.,  p.  26, paragraph  before
Theorem 16, of \cite{V}).   We have assumed that $\sum _{j=1}^{n+1}\mu
_j ( = 2-\mu _{\infty})$  is not an  integer, so $\prod  x_j \neq
1$. Therefore, by Proposition 19 of \cite{V}, $G({\mathcal O}_F)$ acts
irreducibly.  By  Lemma \ref{definite}, $G({\mathcal  O}_F)$ is finite
if and  only if $h$ is  totally anisotropic.  We must  then prove that
the condition of the anisotropy  of $h$ is equivalent to the condition
(\ref{*}). \\

Let $\det(h_j)$  be the $j \times  j$ principal minor of  the form $h$
obtained  by  specializing  to  the  $t_i$,  for  some  fixed  $s  \in
(\Z/d\Z)^*$.  By Lemma \ref{determinant}, the determinant of $h_j$ is
\[\frac{1-t_1t_2\cdots  t_{j+1}}{(1-t_1)\cdots (1-t_{j+1})}.\]  It
is easily seen, in view of (\ref{sin}), that this determinant is
\[  \det  (h_j)=  \frac{i^j}{2^j}  \frac{\sin  (\frac{\pi  (k_1+\cdots
k_{j+1})s}{d})}{\prod _{i=1}^{j+1} \sin (\frac{\pi k_is}{d})} .  \] As
in the proof of Lemma \ref{determinant}, we have
\[\b  _j=   \frac{\det(h_{j+1})  \det(h_{j-1})}{\det(h_j)^2}=  \big  (
\frac{\sin  (\frac{\pi   (k_1+\cdots  k_{j+2})s}{d})\sin  (\frac  {\pi
(k_1+\cdots             k_j)s}{d})}            {\sin            ^2(\pi
\frac{(k_1+\cdots+k_{j+1})s}{d})}\big      )      \frac{\sin(\frac{\pi
k_{j+1}s}{d})}{  \sin (\frac{\pi  k_{j+2}s}{d})}.\] If  $x$ is  not an
integer  then  the  sign  of  $\sin  (\pi x)$  is  simply  the  number
$(-1)^{[x]}$. Therefore, the sign of $\b _j$ is
\[              (-1)^{[\frac{(k_1+\cdots              +k_{j+2})s}{d}]-
[\frac{(k_1+\cdots+k_j)s}{d}]-
[\frac{k_{j+1}s}{d}]-[\frac{k_{j+2}s}{d}]}.\] Since, for all $x,y,z\in
\R$, we have
\[  [x+y+z]-[x]-[y]-[z]= [\{x\}+\{y\}+\{z\}],\]  it  follows that  the
sign  of  $\b _j$  is  just the  number  $\e_j(s)$.   Hence, by  Lemma
\ref{betadef}, the  anisotropy of $h$  is equivalent to  the condition
that $\e _j(s)=1$, for all $j$. This is exactly condition (\ref{*}).
\end{proof}

\begin{lemma} 
  Condition (\ref{*}) is equivalent to condition (\ref{SS}).
\end{lemma}

\begin{proof}  Assume that  condition (\ref{*})  holds.   Applying the
condition with $j = 1$, we  must have $\{ \frac {k_1s}{d} \}+ \{ \frac
{k_2s}{d}\}  \neq   1$,  for  all   $s  \in  (\Z/d\Z)^*$.    For  each
$\overline{s}$ in  the quotient  group $(\Z /d\Z)^*/\{\pm  1\}$, there
are two  representatives $s$  and $d-s$ in  the group  $(\Z/d\Z)^*$ of
units mapping  to $\overline{s}$.  We  consider the numbers $a  = a_s=
[\{ \frac {k_1s}{d}  \}+ \{ \frac {k_2s}{d}\}]$ and  $ b= b_{d-s}= [\{
\frac  {k_1(d-s)}{d}   \}+  \{  \frac   {k_2(d-s)}{d}\}]$.  Since  the
fractional part  of $x$ does not  change if $x$ is  replaced by $x+m$,
for         $m\in         \Z$,         it         follows         that
$b=[\{\frac{-k_1s}{d}\}+\{\frac{-k_2s}{d}                          \}]=
[(1-\{\frac{k_1s}{d}\})+(1-\{\frac{k_2s}{d}\})] = 2-(a+1)= 1-a$.  Thus
one of the  numbers $a,b$ is zero since $0\leq a,b  \leq 1$. We choose
the representative $s$ so that $a = a_s=0$. \\

We now prove  by induction on $j\leq n$ that, for  this choice of $s$,
the     integral    part     of     $\alpha_{j+1}    \stackrel{def}{=}
\{\frac{k_1s}{d}\}+\cdots + \{\frac{k_{j+1}s}{d}\}$  is zero. The case
$j =  1$ was just  treated. Applying this  with $j=n$ will  then prove
that condition (\ref{*}) implies condition (\ref{SS}). \\
  
We now claim that condition (\ref{*}) is equivalent to 
\[[\a _{j+1}]  \equiv [\alpha _{j-1}] \mod  ~2, \] for $2  \leq j \leq
n$.  Indeed, if $m \leq \alpha_{j-1}  < m+1$, for some integer $m \geq
0$, then $\nu_{j-1} = \alpha_{j-1}  - m$.  Thus, $\nu_{j-1} +\mu_{j} +
\mu_{j+1}$ lies in  $(0,1)$ or $(2,3)$ if and  only if $\alpha_{j+1} =
\alpha_{j-1} + \mu_j + \mu_{j+1}$ lies in $(m,m+1)$ or $(m+2,m+3)$, so
that its  integral part  is either the  same as,  or 2 more  than, the
integral part $m$ of $\alpha_{j-1}$. \\

By  induction, we  may  assume that  $[\a  _k]=0$ for  all $k\leq  j$;
therefore, by the above congruence,  we have $[\a _{j+1}]\equiv 0 \mod
~2$. \\

On  the  other hand,  $[\alpha  _{j+1}]=[\alpha  _j]+ [\{\alpha  _j\}+
\{\frac{k_{j+1}s}{d}\}]$.   Since by  induction,  $[\alpha _j]=0$,  it
follows that  $[\alpha _{j+1}]=[\{  \a _j\}+ \{  \frac{k_{j+1}s}{d} \}
]$. Being the integral part of a sum of two numbers in the closed open
interval $[0,1)$,  the latter is at most  one and hence $(0  \leq ) \>
[\a  _{j+1}]\leq 1$.  The conclusion  of the  preceding  paragraph now
implies that  $[\a _{j+1}]=0$,  completing the induction  step.  Hence
condition (\ref{SS}) follows. \\

Conversely,  if  condition  (\ref{SS})  holds, then  all  the  numbers
$[\alpha  _j]$ are  zero,  and  hence the  numbers  $\e_j(s)$ are  all
$1$. This is condition (\ref{*}).
\end{proof}

We can now prove Theorem~\ref{numbertheoreticmonodromy}.

\begin{proof}  Since, by the  assumption on  $\mu _{\infty}$,  we have
$\sum \mu _j\notin \Z$, it follows by Proposition 19 of \cite{V}, that
$\rho _d$ is  irreducible; since (again by \cite{V},  Corollary 3) the
monodromy representation $M_d$ is a  quotient of $\rho _d$, it follows
that $M_d$ is the representation $\rho _d$. \\

The   monodromy   representation (being  the   specialised   Gassner
representation $\rho  _d$) has  image in $U(h)({\mathcal  O}_F)$ where
$h$  is  as  above  (this  is  in subsection  4.1  of  \cite{V}).   By
Lemma~\ref{definite},   the   image  is   finite   if   and  only   if
$U(h)({\mathcal  O}_F)$  is  finite.  So,   by  %and  only  if  $h$  is
%anisotropic at all archimedean  places of $F$, and 
Lemma~\ref{finite},
the image is finite if and  only if condition (\ref{*}) holds.  By the
above lemma, this is equivalent to condition (\ref{SS}).
\end{proof}

\section{Algebraic Lauricella Functions} \label{lauricella}

We list some corollaries to Theorem~\ref{numbertheoreticmonodromy}. We
assume as before, that  $\mu _j=\frac{k_j}{d}$ and $\mu _{\infty}$ are
rational  and  not  integral.    The  corollaries  below follow  from  the
finiteness of monodromy and  the observation that the Lauricella $F_D$-functions 
are  the period integrals associated to  homology classes in
the curve  $C$ whose affine part is  given by $y^d=(x-z_1)^{k_1}\cdots
(x-z_{n+1})^{k_{n+1}}$.  \\

\begin{corollary} The Lauricella $F_D$ -function 
\[F_D(a_1,  \cdots, a_{n+1})=\int  _{a_i}^{a_j} \frac{dx}{y},\]  is an
algebraic function of the variables $a_1, \cdots, a_{n+1}$ if and only
if the condition (\ref{SS}) holds for the numbers $(n,d,k_i)$.
\end{corollary}

The following corollary is to be read up to
equivalence of the $\mu _j$ as defined in the Introduction.

\begin{corollary} If  $n+1\geq 6$, then  the function $F_D(a_1,\cdots,
a_{n+1})$ is not algebraic. \\

If $n+1=5$,  then $F_D$ is algebraic if  and only if $d=6$  and all the
$k_i$ are equal to $1$. \\

If $n+1=4$, then  $F_D$ is algebraic if and only if  $d=6$ and all the
$k_i$ are equal to $1$; or else, all but one of the $k_i$ are equal to
$1$ and one of the $k_i=2$. \\

If $n+1=3$, and if $F_D$ is algebraic, then $d=2m$ and $k_1=k_2=p$ and
$k_3=m-p$, or else $d,k_i$ lie in a finite list, with $d\leq 60$.
\end{corollary}

%In  \cite{Coh-Wol}  and  \cite{Bod},  the complete  classification  of
%Lauricella functions which are algebraic  is given. This list can also
%be recovered from the main result of the present paper. 

\end{document}